\documentclass[article]{svjour3}
\usepackage{moreverb}
\usepackage{subcaption}
\newcommand\BibTeX{{\rmfamily B\kern-.05em \textsc{i\kern-.025em b}\kern-.08emT\kern-.1667em\lower.7ex\hbox{E}\kern-.125emX}}
\usepackage[english]{babel}
\usepackage{pgf}
\usepackage{tikz}
\usepackage{lineno}
\usepackage[utf8]{inputenc}
\usepackage{amsmath,amssymb,amsfonts,mathabx}
\usepackage{enumerate}
\usepackage{blkarray}
\usepackage{cite}
\usepackage{subcaption}
\usepackage{mathtools}
\usepackage{algorithm}
\usepackage{enumerate}
\usepackage[noend]{algpseudocode}
 \makeatletter
\def\BState{\State\hskip-\ALG@thistlm}
\makeatother
\usetikzlibrary{calc,fit,matrix,arrows,automata,positioning,patterns}
\usepackage{hyperref}
\usepackage{xcolor}
\hypersetup{
    colorlinks,
    linkcolor={blue!50!black},
    citecolor={blue!50!black},
    urlcolor={blue!80!black}
}
\newcommand{\B}[1]{\mbox{\boldmath $#1$}}

\newtheorem{assumption}{ASSUMPTION}

\newcommand{\ZZ}{\mathbb{Z}}
\newcommand{\RR}{\mathbb{R}}
\newcommand{\NN}{\mathbb{N}}
\newcommand{\CC}{\mathbb{C}}
\newcommand\ba{\mathbf{a}}
\newcommand\bp{\mathbf{p}}
\newcommand\ie{{\it\thinspace i.e.}}

    \newcommand*\patchAmsMathEnvironmentForLineno[1]{%
       \expandafter\let\csname old#1\expandafter\endcsname\csname #1\endcsname
       \expandafter\let\csname oldend#1\expandafter\endcsname\csname end#1\endcsname
       \renewenvironment{#1}%
          {\linenomath\csname old#1\endcsname}%
          {\csname oldend#1\endcsname\endlinenomath}}%
    \newcommand*\patchBothAmsMathEnvironmentsForLineno[1]{%
       \patchAmsMathEnvironmentForLineno{#1}%
       \patchAmsMathEnvironmentForLineno{#1*}}%
    \AtBeginDocument{%
    \patchBothAmsMathEnvironmentsForLineno{equation}%
    \patchBothAmsMathEnvironmentsForLineno{align}%
    \patchBothAmsMathEnvironmentsForLineno{flalign}%
    \patchBothAmsMathEnvironmentsForLineno{alignat}%
    \patchBothAmsMathEnvironmentsForLineno{gather}%
    \patchBothAmsMathEnvironmentsForLineno{multline}%
    }
\usepackage{pgfplots}
\usepackage{hyperref}

\DeclareMathOperator{\diag}{\mathcal D}

\usepackage{siunitx}
\title{Bezout-like polynomial equations associated with dual univariate interpolating subdivision schemes}
\titlerunning{Bezout-like equations for dual interpolating subdivision}
\author{Luca Gemignani \and Lucia Romani \and Alberto Viscardi}
\institute{ Luca Gemignani \at
	Dipartimento di Informatica, Universit\`{a} di Pisa,\\
	Largo Bruno Pontecorvo, 3 - 56127 Pisa, Italy\\
	\email{luca.gemignani@unipi.it}
	\and
	Lucia Romani, Alberto Viscardi \at Dipartimento di Matematica, Alma Mater Studiorum Universit\`{a} di Bologna,\\
	Piazza di Porta San Donato, 5 - 40126 Bologna, Italy\\
	\email{lucia.romani@unibo.it}, \email{alberto.viscardi@unibo.it}
}
\date{}%\today

\begin{document}
%\linenumbers
\maketitle

\begin{abstract}
  The algebraic characterization of  dual univariate interpolating subdivision schemes is investigated.
  Specifically,  we provide a constructive approach for finding dual univariate interpolating subdivision schemes based on the
  solutions of certain associated polynomial equations.  The  proposed approach also makes possible to identify
  conditions for the existence of the sought schemes.

\end{abstract}



\section{Introduction}

Subdivision schemes are useful tools for the fast generation of graphs of functions, smooth curves and surfaces by the application of iterative refinements to an initial set of discrete data.
The major fields of application of subdivision schemes are  Computer Graphics and Animation, Computer-Aided Geometric Design and Signal/Image Processing, but a further
motivation for their study is also their close relation to multiresolution analysis and wavelets.
The last connection was especially investigated in the case of interpolating subdivision schemes and
it was pointed out that interpolatory subdivision schemes of Dubuc-Deslauriers  \cite{MR982724} are connected to
orthonormal wavelets of Daubechies \cite{CDbook,Micchelli1996InterpolatorySS}.
Interpolating subdivision schemes were also deeply studied because considered very efficient in representing smooth curves and surfaces passing through a given set of points.
In fact, after five or six subdivision iterations only, they are capable of providing the refined set of points
needed to represent on the screen the desired smooth limit shape interpolating the given data.
The main properties of interpolating subdivision schemes were investigated over the past 20 years by several researchers
(see, e.g., \cite{DUBUC1986185, Dyn02, GAVHIMOLEFE2019354}) and many approaches were proposed to design their refinement rules.
However, as far as we are aware, before the papers  \cite{LUCIA,RV}, no one ever tried to construct interpolating subdivision schemes
that do not satisfy the stepwise interpolation property and are thus not defined via refinement rules that at each stage of the iteration
leave the previous set of points unchanged.
Stepwise interpolating subdivision schemes - also known as primal interpolating subdivision schemes \cite{MR982724,MR2008967} -
are defined by finite subdivision masks of odd width that contain as a special submask the sequence
${\boldsymbol \delta}=\{\delta_{0,j}, \, j \in \mathbb{Z} \}$.
Differently, members of the most recently introduced class of non-stepwise interpolating subdivision schemes --also known as dual interpolating subdivision schemes--
are characterized by finite masks with an even number of entries that do not necessarily contain as a special submask the ${\boldsymbol \delta}$ sequence.
One of the contributions of \cite{RV} was to show that, under some suitable/auxiliary assumptions, the  coefficients of the subdivision mask of a dual interpolating
scheme can be (possibly) determined by the solution  of an associated rectangular linear system.
This  system can be clearly inconsistent for some choices of  input data and/or size (length) of the mask.
For a given input data set  the approach taken in \cite{RV} consists of an exhaustive analysis
of the associated linear systems of increasing sizes   in order to identify possible consistent configurations.

In this paper we pursue  a different method for  constructing dual interpolating subdivision schemes based on the reduction of the matrix formulation into a functional setting to
solving a certain Bezout-like polynomial equation.  The method makes possible to address the consistency issues by detecting  suitable conditions on the input data which guarantee the
existence of  a dual interpolating scheme.  Additionally, it  yields a  full  characterization of the set of solutions which can be exploited to fulfil additional demands and
properties of the solution mask.
From the point of view of applications, such a computational approach allows the user to meet specific requests in terms of
polynomial reproduction, support size and regularity. Even though a general result concerning convergence and/or
smoothness of a dual interpolating subdivision scheme is not yet available, in all the considered specific cases the analysis can be
performed by using ad-hoc techniques. Illustrative examples and comparisons with existing primal interpolating schemes are provided and discussed.

\section{Background and notation}\label{Back}

In this section we briefly recall some needed background on subdivision schemes of arbitrary arity $m \in \NN$, $m \geq 2$.

Any linear, stationary subdivision scheme is identified by a \emph{refinement mask} ${\ba}:=\left \{ a_{i}\in \RR,\ i\in \ZZ\right\}$
that is usually assumed to have finite support, \ie\ to satisfy
$a_i=0$ for $i\not\in [-L, L]$ for suitable $L > 0$.

The \emph{subdivision scheme} identified by the mask $\ba$ consists of the subsequent application of the \emph{subdivision operator}
$$
S_{\ba} \ :\ \ell(\ZZ) \rightarrow
\ell(\ZZ) \ ,\qquad \displaystyle{(S_{\ba}\
{\bp})_{i}\;:=\;\sum_{j \in\ZZ} a_{i-m j}\ {p}_{j}},\quad i \in
\ZZ \ ,
$$
which provides the linear rules determining the successive refinements of the initial sequence of discrete data
$\bp:=\left(p_i\in\RR, i\in \ZZ\right) \in \ell(\ZZ)$.
Introducing the notation ${\bp}^{(0)}:={\bp}$, we can thus describe
the subdivision scheme as an iterative method that at the $k$-th step generates
the refined scalar sequence
\begin{equation}\label{sscheme}
{\bf \bp}^{(k+1)}\;:=\;S_{\ba} \, {\bp}^{(k)}, \qquad  k\ge 0.
\end{equation}
Attaching the data $p_i^{(k)}$ generated at the $k$-th step to the parameter values $t^{(k)}_i$ with
$$t^{(k)}_i\;<\;t_{i+1}^{(k)}, \qquad \hbox{and} \qquad t_{i+1}^{(k)}-t_i^{(k)}\;=\;m^{-k},\qquad k\;\ge\; 0$$
(these are usually set as $t_i^{(k)}:=m^{-k} i$) we see that the subdivision process generates denser and
denser sequences of data so that a notion of convergence can be established by taking into account the piecewise
linear function $P^{(k)}$ that interpolates the data, namely
\[
  P^{(k)}(t_i^{(k)}) \;=\; p_i^{(k)}, \qquad
  P^{(k)}|_{[t_i^{(k)},t_{i+1}^{(k)}]} \in \Pi_1, \qquad
  i\in\ZZ,\quad k\geq0,
\]
where $\Pi_1$ is the space of linear polynomials.
If the sequence of the continuous functions $\{P^{(k)},\ k\ge 0\}$ converges uniformly, then we denote its
limit by
\[
  f_\bp \;:=\; \lim_{k\to\infty} P^{(k)}
\]
and say that $f_\bp$ is the \emph{limit function} of
the subdivision scheme based on the rule (\ref{sscheme}) for the data $\bp$ \cite{MR1079033}. When $\mathbf{p}={\boldsymbol\delta}$, $f_{\boldsymbol\delta}$ is called \emph{basic limit function}.\\
The analysis of convergence of a subdivision scheme can be accomplished by studying the properties of the so-called \emph{symbol} of the subdivision mask \cite{MR1172120}.
The symbol of a finitely supported sequence $\ba$ is defined as the Laurent polynomial
$$
a(z) \;:=\; \sum_{i\in \ZZ}a_i \, z^i,  \qquad z \in \CC\setminus \{0\}.
$$
Besides convergence and smoothness, many other properties of a subdivision scheme, like polynomial generation and
reproduction, can be checked by investigating algebraic conditions on the subdivision symbol \cite{MR2775138}.
While the term \emph{polynomial generation} refers to the capability of the subdivision scheme of providing polynomials as limit functions,
with \emph{polynomial reproduction} we mean the capability of a subdivision scheme of reproducing in the limit exactly the same polynomial from which the data are sampled.
The property of polynomial reproduction is very important since strictly connected to the approximation order of the subdivision scheme and to its regularity \cite{CHOI2006351,MR2474706}.
With respect to the capability of reproducing polynomials up to a certain degree, the standard parametrization (corresponding to the choice $t_i^{(k)}:=m^{-k} i$, $i \in \ZZ$)
is not always the optimal one. Indeed, the choice $t_i^{(k)}:=m^{-k} (i+\sigma/(m-1))$
with $\sigma=a^{(1)}(1)/m$, turns out to be the recommended selection \cite{MR3071114}.
The subdivision schemes for which  $\sigma \in \ZZ$ are termed \emph{primal}, whereas the ones for which $\sigma \in (2\ZZ+1)/2$ are called \emph{dual}.
The target of this work are dual schemes. While dual approximating schemes were investigated extensively (see, e.g., \cite{MR3071114,DUBUC2011966} and references therein),
to the best of our knowledge dual interpolating schemes were only considered in the recent papers \cite{LUCIA,RV}.

\section{Basic reductions}

The aim of this section is to investigate the algebraic characterization of univariate dual interpolating subdivision schemes of arity $m$.
According to the results shown in \cite{RV}, the construction of such schemes requires as input the desired degree of polynomial reproduction and some samples of the
resulting basic limit function $f_{\boldsymbol \delta}$. A similar procedure was investigated in \cite{DeVilliers2016759,MR1790328}, where the samples of
the basic limit function at the integers were required: here instead the samples at the integers are fixed to be the $\boldsymbol\delta$ sequence and
information about the samples at the half-integers are required.\\
More specifically,  in \cite{RV} it is seen that   taking Fourier transforms on both sides of the refinement equation for the  basic limit function
  $f_{\boldsymbol \delta}$  allows one to describe
the mask of dual interpolatory schemes  in a matrix setting in terms of the solution of certain bi-infinite Toeplitz-like linear systems
in banded form. In this paper we exploit the interplay between the functional and the matrix settings into more details. In particular, from the matrix setting we come back to
the functional  one by relying upon the  connection of Toeplitz-like systems with corresponding  Bezout-like polynomial equations.   This  connection yields
a constructive approach to determine the associated symbols. Moreover, the proposed approach also makes possible to identify conditions
for the existence of the sought dual interpolatory schemes\\
In the following, to simplify the presentation, we distinguish between odd and even arity.

\subsection{The odd arity case}

Now let us consider the solution of the linear system $(35)$ in \cite{RV} for the case where
$m=2\ell+1$ is an odd integer.   The system is defined as follows:
\begin{equation}\label{sys}
  M \B a\;=\;\B c, \qquad M\;=\;(\mu_{i,j})_{i,j\in \mathbb Z}, \qquad  \B c \;=\;(c_i)_{i\in \mathbb Z}
  \end{equation}
  where
  \[
  \mu_{i,j}=\left\{\begin{array}{cl}
  \varphi \left(\displaystyle\frac{i+1}{2}-j \right), & \mbox{ if } i\in 2m\mathbb Z,\\
  \\
  1,  & \mbox{ if } i\in m(2\mathbb Z+1), \ j=\frac{i+1}{2},\\ \\
  0, & \mbox{ otherwise,}
  \end{array}\right. 
  \]
  \[
  c_i=\left\{\begin{array}{cl}
  1,  & \mbox{ if } i=0,\\ \\
  \varphi \left(\displaystyle\frac{i}{2m} \right),  & \mbox{ if } i\in m(2\mathbb Z+1),\\ \\
   0, & \mbox{ otherwise,}
  \end{array}\right.
  \]
  and $\varphi\colon (2\mathbb Z+1)/2 \rightarrow \mathbb R$ is  a given fixed function. By  suppressing zero rows in
  both $M$ and $\B c$ we obtain  the equivalent linear system
 \begin{equation}\label{sys1}
   \widehat M \B a\;=\;\widehat {\B c}, \qquad
   \widehat M\;=\;(\widehat \mu_{i,j})_{i,j\in \mathbb Z}, \qquad  \widehat {\B c} \;=\;(\widehat c_i)_{i\in \mathbb Z}
  \end{equation}
  where
  \[
  \widehat  \mu_{i,j}=\left\{\begin{array}{cl}
  \varphi \left(\displaystyle\frac{im+1}{2}-j \right),&\mbox{ if }  \mod(i,2)=0,\\ \\
  1, & \mbox{ if }  \mod(i,2)=1, \ j=\frac{im+1}{2},
  \end{array}\right.
  \]
  \[
  \widehat c_i=\left\{\begin{array}{cl}
  1,  & \mbox{ if } i=0,\\ \\
  \varphi \left(\displaystyle\frac{i}{2} \right), & \mbox{ if }  \mod(i,2)=1,\\ \\
   0, & \mbox{ otherwise.}
  \end{array}\right.
  \]
The interplay between  computations with polynomials and Toeplitz-like matrices can  be exploited to recast the solution of the linear system \eqref{sys} in terms of
  solving an associated Bezout-like polynomial equation. Indeed  from the proof of Theorem 4.1 in \cite{RV} one deduces that  the entries of the unknown
  vector $\B a$ satisfy
  \begin{equation}\label{sys1p}
    \left\{\begin{array}{ll}
    \displaystyle\sum_{\alpha\in m(2\mathbb Z+1)}\varphi \left(\frac{\alpha}{2m} \right)z^{\alpha}= \sum_{\alpha\in m(2\mathbb Z+1)} a_{\frac{\alpha+1}{2}} z^\alpha\\ \\
    \displaystyle 1=\sum_{\alpha\in 2m\mathbb Z} \sum_{\beta\in \mathbb Z} a_{\beta} \, \varphi \left(\frac{\alpha+1}{2}-\beta \right)z^{\alpha} \end{array}\right.
  \end{equation}
  which implies
  \begin{equation}\label{sys1pp}
    \left\{\begin{array}{ll}
    a_{mi+\frac{m+1}{2}}=\varphi \left(\frac{2i+1}{2} \right), \qquad  i\in \mathbb Z, \\ \\
    \displaystyle 1-\sum_{\alpha\in 2m\mathbb Z}\; \sum_{\beta\in m\mathbb Z+\frac{m+1}{2}} a_{\beta} \varphi \left(\frac{\alpha+1}{2}-\beta \right) z^{\alpha}=\\ \\
    \displaystyle\qquad\qquad\qquad\sum_{\alpha\in 2m\mathbb Z} \sum_{\substack{\beta\in \mathbb Z\\ \mod(m,\beta)\neq\frac{m+1}{2}} } a_{\beta} \varphi \left(\frac{\alpha+1}{2}-\beta \right) z^{\alpha}.
     \end{array}\right.
  \end{equation}
  The system \eqref{sys1pp} can be rewritten into a  more compact form   by using   the decomposition of
  $a(z)=\sum_{i\in\mathbb Z}a_iz^i$   that involves the sub-symbols of the scheme   given by
  \begin{equation}\label{dec}
    a(z)=\sum_{i=0}^{m-1}a_i(z^m) z^i, \qquad a_\ell(z)=\sum_{i\in \mathbb Z} a_{mi+\ell}z^i,\quad \ 0\leq \ell \leq m-1.
  \end{equation}
  Let us introduce the corresponding decomposition of the Laurent polynomial
  $\phi(z)= \sum_{\ell\in \mathbb Z}  \varphi \left(\displaystyle\frac{1}{2}+\ell \right) z^{\ell}$ defined by
  \begin{equation}\label{dec1}
   \phi(z)=\sum_{i=0}^{m-1}\phi_i(z^m) z^{-i}, \qquad \phi_\ell(z)=\sum_{i\in \mathbb Z}\varphi \left(\frac{2mi+1}{2}-\ell \right)z^i, \quad \ 0\leq \ell \leq m-1.
  \end{equation}
  The first equation  of \eqref{sys1pp} determines $a_{\frac{m+1}{2}}(z)$. Then the second equation can  be read as follows
  \[
    1-a_{\frac{m+1}{2}}(z^m) \phi_{\frac{m+1}{2}}(z^m)\;=\;\sum_{i=0, i\neq \frac{m+1}{2}}^{m-1} a_i(z^m) \phi_i(z^m)
    \]
    or, equivalently,
    \begin{equation}\label{dec11}
      1-a_{\frac{m+1}{2}}(z) \phi_{\frac{m+1}{2}}(z)\;=\;\sum_{i=0, i\neq \frac{m+1}{2}}^{m-1} a_i(z) \phi_i(z).
    \end{equation}
    Our computational task is  therefore reduced to compute a  Laurent polynomial $a(z)$ defined as in \eqref{dec}  satisfying the
    Bezout-like polynomial equation \eqref{dec11}.  It is quite natural for convergence and reproducibility issues
    to  impose some other constraints of the form
    \begin{equation}\label{addit}
      \begin{array}{ll}
        a_i(1)=1, \quad 0\leq i\leq m-1, \\ \\
        a(z)= \displaystyle \left(\frac{1+z+\ldots +z^{m-1}}{m}\right)^d b(z).
      \end{array}
    \end{equation}
    Our proposed construction of such a polynomial $a(z)$ works under some additional assumptions on the  input data
    $\{\varphi((2k+1)/2)\}_{k=-\kappa}^{\kappa-1}$ encoded in the function $\phi(z)$.  More  specifically:
	\begin{assumption}\label{a1}: We  suppose that   $1-z\phi(z^2)=(z-1)^d\gamma(z)$ for a  certain $\gamma(z)\in \mathbb R[z,z^{-1}]$  the ring of Laurent polynomials in $z,z^{-1}$ over $\mathbb R$.
    \end{assumption}
    \begin{assumption}\label{a2}: We suppose that  $\phi_i(z)\in \mathbb R[z,z^{-1}]$,  $0\leq i\leq m-1$, $i\neq (m+1)/2$ are relatively prime, i.e. they have no common zeros.
    \end{assumption}
    Under these assumptions our composite  approach for computing $a(z)$  proceeds  by the following steps.

    \subsection{The proposed approach}

   The first step consists of  determining the values
    $a_i^{(s)}(1)$, $0\leq i\leq m-1$, $s=0,\ldots d-1$. From \eqref{addit} one gets immediately  $a_i^{(0)}(1)=a_i(1)=1$, $0\leq i\leq m-1$.
   The first equation  of \eqref{sys1pp} implies $\phi(z)= a_{\frac{m+1}{2}}(z)$ and, hence,  from Assumption \ref{a1}
\begin{equation}\label{eq1}
    1-z \phi(z^2) \;=\; 1-z a_{\frac{m+1}{2}}(z^2)\;=\;(z-1)^d \gamma(z), \qquad \gamma(z)\in \mathbb R[z,z^{-1}].
    \end{equation}
This equation sets the values attained by the function  $\phi(z)= a_{\frac{m+1}{2}}(z)$ and its derivatives  at the point 1.
    \begin{theorem}
      If $\phi(z)$ satisfies  \eqref{eq1} then it holds
      \[
      \left\{ \begin{array}{ccl}
        \phi(1)&=&1,\\ \\
        \phi^{(k)}(1)&=&(-1)^k\frac{(2k-1)!!}{2^{k}}, \qquad 1\leq k\leq d-1.
      \end{array}
      \right.
      \]
    \end{theorem}
    \begin{proof}
      Substituting $z=\sqrt{w}$ in \eqref{eq1}, we get
      \[
      \phi(w)-w^{-1/2}\;=\;\frac{(1-\sqrt{w})^d (-1)^{d+1}\gamma(\sqrt{w})}{\sqrt{w}}.
      \]
The proof easily follows by differentiating  this relation at $w=z=1$. \ $\square$
\end{proof}
The remaining unknowns $a^{(s)}_i(1)$, $0\leq i\leq m-1$, $i\neq (m+1)/2$, $s=1,\ldots d-1$,  are computed
by solving the linear system obtained by differentiation
of \eqref{addit}. Specifically, by differentiating $s$ times the expression  of $a(z)$ in \eqref{dec}   with respect to the variable $z$  we find that
\begin{equation}\label{composder}
\scalebox{0.95}{$\displaystyle
a^{(s)}(z)=\sum_{i=0}^{m-1}\sum_{p=0}^s\frac{a_i^{(p)}(z^m)}{p!}\left(\sum_{j=max\{s-i,p\}}^s \binom{s}{j} A_{j,p}(z) \frac{i!}{(i-(s-j))!} z^{i-(s-j)}\right),
$}
\end{equation}
where $A_{j,p}(z)$ are polynomials defined by  Hoppe's formula for derivation of  composite function   according to
\[
A_{j,p}(z)\;=\;\sum_{\ell=0}^j\binom{p}{\ell}(-f(z))^{p-\ell} \frac{d^j}{dz^j}(f(z))^\ell, \qquad f(z)\;=\;z^m.
\]
If  $\xi_k=e^{2\pi \textrm{i} k/m}$, $1\leq k\leq m-1$, are the $m$-th roots of unity, then from \eqref{addit} it follows that $a^{(s)}(\xi_k)=0$,  $s=0,\ldots d-1$,
  $1\leq k\leq m-1$.  In the view of  \eqref{composder}  this implies that   the values $a^{(s)}_i(1)$, $0\leq i\leq m-1$, $i\neq (m+1)/2$, $s=1,\ldots d-1$, can be
  computed recursively  by solving
  \[
  \sum_{i=0}^{m-1}\sum_{p=0}^s\frac{a_i^{(p)}(1)}{p!}\left(\sum_{j=max\{s-i,p\}}^s \binom{s}{j} A_{j,p}(\xi_k) \frac{i!}{(i-(s-j))!} \xi_k^{i-(s-j)}\right)\;=\;0,
  \]
  with $1\leq k\leq m-1$.
 The system can be expressed  in matrix form as
\begin{equation} \label{eq:lin_sys_1}
	\scalebox{0.9}{$m^s\diag\left(\xi_1^{(m-1)s}, \ldots, \xi_{m-1}^{(m-1)s}\right) \mathcal V(\xi_1, \ldots, \xi_{m-1})\left[ a_0^{(s)}(1), \ldots, a_{m-1}^{(s)}(1)\right]^T\;=\;\B b_s,$}
\end{equation}
where
\begin{align*}
(\B b_s)_k& = -\sum_{i=0}^{m-1}\sum_{p=0}^{s-1}\frac{a_i^{(p)}(1)}{p!}\left(\sum_{j=max\{s-i,p\}}^s \binom{s}{j} A_{j,p}(\xi_k) \frac{i!}{(i-(s-j))!} \xi_k^{i-(s-j)}\right)\\ \\ 
&\qquad\qquad\qquad-\frac{a_{\frac{m+1}{2}}^{(s)}(1)}{s!} A_{s,s}(\xi_k) \xi_k^{\frac{m+1}{2}}, \qquad 1\leq k\leq m-1.
\end{align*}
Here $\diag(\B v)$,  $\B v=\left[v_1, \ldots, v_{m-1}\right]^T$, is the diagonal matrix with diagonal entries $v_k$, $1\leq k\leq m-1$,  and  $\mathcal V(\xi_1, \ldots, \xi_{m-1})$ is the Vandermonde matrix with nodes $\xi_k$, $1\leq k\leq m-1$.
Since $\xi_k$,  $1\leq k\leq m-1$, are distinct and non-zero,  the coefficient matrix is nonsingular and $a^{(s)}_i(1)$, $0\leq i\leq m-1$, $i\neq (m+1)/2$, are uniquely
determined.

Once the quantities $a^{(s)}_i(1)$, $0\leq i\leq m-1$, $s=0,\ldots d-1$,
are calculated then  the sub-symbols $a_i(z)$, $ 0\leq i\leq m-1$,  $ i\neq (m+1)/2$,  can be represented as follows
    \begin{equation}\label{eq:aiz_odd}
    a_i(z)\;=\;1 + \sum_{j=1}^{d-1}\frac{a^{(j)}_i(1)}{j!} (z-1)^j + (z-1)^d \widehat a_i(z) \;=\; \widecheck a_{i}(z) +(z-1)^d \widehat a_i(z),
    \end{equation}
    for suitable   $\widehat a_i(z)\in \mathbb R[z, z^{-1}]$.
    This representation is exploited in the second step to find  a solution of \eqref{dec11}.  Let us introduce the  truncated representation  $\widecheck a(z)$ of the symbol $a(z)$, that is,
    \[
   \widecheck a(z)\;=\; \sum_{i=0,\;  i\neq \frac{m+1}{2}}^{m-1}\widecheck a_{i}(z^m) z^i \;+\; a_{\frac{m+1}{2}}(z^m) z^{\frac{m+1}{2}}.
    \]
    First of all, we notice that
     in the view of \eqref{dec11}  the sub-symbols  of the function  $\phi(z)$  should  fulfil  the  compatibility relations  obtained by differentiating
    \eqref{dec11}  at the point $z=1$.
    %%%%%%%%%%%%%%%%%%%%%%%%%%%%%%%%%%%%%%%%%
    Specifically,  by setting
    \[
    \theta(z)\;=\; 1-a_{\frac{m+1}{2}}(z) \phi_{\frac{m+1}{2}}(z)-\sum_{i=0,\; i\neq \frac{m+1}{2}}^{m-1} \widecheck a_{i}(z) \phi_i(z),
    \]
    we require that
    \[
    \theta^{(s)}(1)\;=\;0, \qquad 0\leq s\leq d-1.
    \]
The following result provides this compatibility for free.
    \begin{theorem}\label{theo2}
      If $\phi(z)= a_{\frac{m+1}{2}}(z)$ satisfies  \eqref{eq1},  then
      the function
      \[
      \theta(z)\;=\;1-a_{\frac{m+1}{2}}(z) \phi_{\frac{m+1}{2}}(z)-\sum_{i=0,\; i\neq \frac{m+1}{2}}^{m-1} \widecheck a_{i}(z) \phi_i(z)
      \]
      is such that
      $\theta^{(s)}(1)=0$ for  $s=0,\ldots d-1$.
    \end{theorem}
    \begin{proof}
      Let us  consider the   auxiliary function $q(z)=z^{-\frac{m+1}{2}} \widecheck a(z^2) z \phi(z^2)$.  From \eqref{eq1}
      it follows that $q(z)=z^{-\frac{m+1}{2}} \widecheck a(z^2) - z^{-\frac{m+1}{2}} \widecheck a(z^2) (-1)^d(1-z)^d \gamma(z)$.  By construction   $\widecheck a(z)$  satisfies relations \eqref{addit}.
      By using the  representation of
      $\widecheck a(z)$ provided by \eqref{addit}   this     gives
      \[
      q(z)\;=\;z^{-\frac{m+1}{2}} \widecheck a(z^2) + \frac{(1-z^m)^d(1+z^m)^d}{(1+z)^d} \widehat \rho(z)
      \]
      with $\widehat \rho(z) \in \RR[z,z^{-1}]$.
      Observe that
      \[
      z^{-\frac{m+1}{2}} \widecheck a(z^2)\;=\;z^{\frac{m+1}{2}}a_{\frac{m+1}{2}}(z^{2m}) + \sum_{i=0,\;i\neq \frac{m+1}{2} }^{m-1}\widecheck a_{i}(z^{2m}) z^{2i-\frac{m+1}{2}},
      \]
      and, hence,
      \begin{equation}\label{new1}
      \scalebox{0.9}{$\displaystyle
       q(z)\;=\;z^{\frac{m+1}{2}}a_{\frac{m+1}{2}}(z^{2m}) + \sum_{i=0,\;i\neq \frac{m+1}{2} }^{m-1}\widecheck a_{i}(z^{2m}) z^{2i-\frac{m+1}{2}} + \frac{(1-z^m)^d(1+z^m)^d}{(1+z)^d} \widehat \rho(z).
      $}
      \end{equation}
     Moreover  it can be easily seen that  the two sets $[0,m-1]\cap \mathbb N$ and  $\{n\in \mathbb N\colon n= 2i-(m+1)/2 \pmod{m}, 0\leq i\leq m-1\}$
     coincide. Besides this,  by direct multiplication of $a(z^2)$ and $\phi(z^2)$, we can write
     \begin{equation}\label{new2}
     \begin{split}
         q(z)\;&=\;z^{\frac{1-m}{2}}\left(\sum_{i=0,\;i\neq \frac{m+1}{2} }^{m-1} \widecheck a_{i}(z^{2m}) \phi_i(z^{2m}) +a_{\frac{m+1}{2}}(z^{2m}) \phi_{\frac{m+1}{2}}(z^{2m})\right) +\\ \\
         &\qquad\qquad\qquad+
         z^{\frac{1-m}{2}}\sum_{i\neq j,\; 0\leq i,j\leq  m-1}z^{2(i-j)}\eta_{i,j}(z^{2m}),
     \end{split}
     \end{equation}
     for suitable Laurent polynomials $\eta_{i,j}(z)\in \RR[z,z^{-1}]$.
     Since  $(1-m)/2\equiv (m+1)/2 \pmod{m}$  the class of integers congruent to $(1-m)/2$  modulo $m$ is $\{n\in \mathbb Z\colon n= (1-m)/2+ \ell m, \ell \in \mathbb Z\}$.  It  follows that  $n=(1-m)/2 +2(i-j)$, $i\neq j, 0\leq i,j\leq  m-1$,
     is such that $n \not\equiv (1-m)/2 \pmod{m}$.  Hence,  by  comparison of classes mod $m$ in \eqref{new1} and \eqref{new2},  we obtain that
     \[
     \begin{split}
       z^ma_{\frac{m+1}{2}}(z^{2m})\;&=\;\sum_{i=0,\;i\neq \frac{m+1}{2}  }^{m-1} \widecheck a_{i}(z^{2m}) \phi_i(z^{2m}) + a_{\frac{m+1}{2}}(z^{2m}) \phi_{\frac{m+1}{2}}(z^{2m})+
       \\ \\&\qquad\qquad\qquad+(1-z^m)^d \widetilde \rho(z), \qquad \widetilde \rho(z) \in \RR[z,z^{-1}].
     \end{split}
     \]
      From \eqref{eq1} this implies that
     \[
     \sum_{i=0,i\neq \frac{m+1}{2}  }^{m-1} \widecheck a_{i}(z^{2}) \phi_i(z^{2}) +a_{\frac{m+1}{2}}(z^{2}) \phi_{\frac{m+1}{2}}(z^{2}) \;=\;1+(1-z)^d\rho(z), \qquad \rho(z) \in \RR[z,z^{-1}]
     \]
     which concludes the proof. \ $\square$ \
    \end{proof}

%%%%%%%%%%%%%%%%%%%%%%%%%%%%%%%%%%%%%%%%%
    By setting
    \begin{equation}\label{eq:thetaz}
    	\theta(z)\;=\;(z-1)^d \widehat \theta(z), \qquad \widehat \theta(z) \in \mathbb R[z,z^{-1}].
    \end{equation}
%\begin{equation}\label{eq:thetaz}
%    \theta(z)\;=\; 1-a_{\frac{m+1}{2}}(z) \phi_{\frac{m+1}{2}}(z)-\sum_{i=0, i\neq \frac{m+1}{2}}^{m-1}\left(1 + \sum_{j=1}^{d-1}\frac{a^{(j)}_i(1)}{j!} (z-1)^j \right)  \phi_i(z),
%    \end{equation}
    from Theorem \ref{theo2}
    %we obtain that
    %\[
   % \theta(z)\;=\;(z-1)^d \widehat \theta(z), \qquad \widehat \theta(z) \in \mathbb R[z,z^{-1}].
   % \]
   it follows that the polynomial corrections $\widehat a_i(z)$, $0\leq i\leq m-1,$  $i\neq (m+1)/2$,   satisfy the Bezout equation
    \begin{equation}\label{bezz1}
      \widehat \theta(z)\;=\;\sum_{i=0,\; i\neq \frac{m+1}{2}}^{m-1}\widehat a_i(z)  \phi_i(z).
    \end{equation}
     Under  Assumption \ref{a2}
     this polynomial equation is solvable   \cite{GS}. In particular,  following  \cite{GS}  every solution of \eqref{bezz1}
     can be written as
     \[
     \widehat a_i(z)\;=\;\widetilde a_i(z)  +\sum_{j=i+1,\;j\neq \frac{m+1}{2}}^{m-1}H_{i,j}(z)\phi_j(z) -\sum_{j=0,\; j\neq \frac{m+1}{2}}^{i-1}H_{j,i}(z)\phi_j(z),
     \]
     where
     \[
     \widehat \theta(z)\;=\;\sum_{i=0,\; i\neq \frac{m+1}{2}}^{m-1}\widetilde  a_i(z)  \phi_i(z)
     \]
     and $H_{i,j}(z)$  is any element of $\mathbb R[z,z^{-1}]$. This general form of the solution can be exploited whenever we look for masks $a(z)$ with additional properties.
     Of great importance for applications is the case where $a(z)$ is required to be symmetric, that is, $a(z)= za(z^{-1})$.  The existence of a symmetric solution can be proved under the auxiliary assumption that 	$\varphi\left(1/2+\ell\right)=\varphi\left(-1/2-\ell\right)$, $\ell \in \mathbb N\cup\{0\}$. %Z$, $\ell\geq 0$.
     In this case   from  \eqref{sys1p}  we  obtain that  $a(z)$ is a solution if and only if $za(z^{-1})$  is a solution, too. By linearity  this implies that $(a(z)+za(z^{-1}))/2$  also  determines a symmetric solution.    If this solution is not of minimal length one can exploit the general form above to   further
       compress the  representation.

     \begin{example}
       	Let us illustrate  our composite approach for the odd case  by means of a computational example. We choose $m=3$, $d=6$ and
       	\[
       		\varphi\left(\frac{1}{2}+\ell\right)\;=\;\left\{\begin{array}{cl}
       			\frac{3}{256}, & \textrm{if } \ell\in\{-3,2\},\\ \\
       			-\frac{25}{256}, & \textrm{if } \ell\in\{-2,1\},\\ \\
       			\frac{75}{128}, & \textrm{if } \ell\in\{-1,0\},\\ \\
       			0, & \textrm{ otherwise.}	
       		\end{array}		\right.
       	\]
       	Thus, according to \eqref{dec1},
       	\[\begin{array}{rcl}
       		\phi(z)&=&\frac{3}{256\,z^3}\;-\;\frac{25}{256\,z^2}\;+\;\frac{75}{128\,z}\;+\;\frac{75}{128}\;-\;\frac{25\,z}{256}\;+\;\frac{3\,z^2}{256} \\ \\
       		&=&\phi_0(z^3)+\phi_1(z^3)z^{-1}+\phi_2(z^3)z^{-2},
       	\end{array}\]
        with
       	\[
       		\phi_0(z)\;=\; \frac{3}{256\,z}\;+\;\frac{75}{128},\qquad
       		\phi_1(z)\;=\;\frac{75}{128}\;+\;\frac{3\,z}{256},\qquad
       		\phi_2(z)\;=\;-\frac{25}{256}\;-\;\frac{25\,z}{256}.
       	\]
       	After solving the linear system \eqref{eq:lin_sys_1}, we have from \eqref{eq:aiz_odd}
       	\[
 			a_0(z)\;=\;\widecheck{a}_0(z)+(z-1)^6\widehat{a}_0(z),\qquad a_1(z)\;=\;\widecheck{a}_1(z)+(z-1)^6\widehat{a}_1(z)
 		\]
 		with
 		\[
 			\widecheck{a}_0(z)\;=\;1+\frac{(z-1)}{6}-\frac{5\,{\left(z-1\right)}^2}{72}+\frac{55\,{\left(z-1\right)}^3}{1296}-\frac{935\,{\left(z-1\right)}^4}{31104}+\frac{4301\,{\left(z-1\right)}^5}{186624},
 		\]
 		\[
 			\widecheck{a}_1(z)\;=\;1-\frac{(z-1)}{6}+\frac{7\,{\left(z-1\right)}^2}{72}-\frac{91\,{\left(z-1\right)}^3}{1296}+\frac{1729\,{\left(z-1\right)}^4}{31104}-\frac{8645\,{\left(z-1\right)}^5}{186624},
 		\]
 		and
 		\[ 	
 			a_2(z)\;=\;\phi(z).	
       	\]
       	To search for compatible $\widehat{a}_0(z)$ and $\widehat{a}_1(z)$, we first compute
       	\[
       		\widehat{\theta}(z)\;=\;\frac{8645\,z^3+215471\,z^2-24300\,z+18225}{15925248\,z^3}
       	\]
       	in such a way that \eqref{eq:thetaz} holds, i.e.,
       	\[
       		(z-1)^6\widehat{\theta}(z)\;=\;1\;-\;a_2(z)\;\phi_2(z)\;-\;\sum_{i=0}^{1}\;\widecheck{a}_i(z)\;\phi_i(z).%\Big( a_i(z)-(z-1)^6\widehat{a}_i(z)\Big) \;\phi_i(z).
       	\]
       	Then we look for particular solutions $\widetilde{a}_0(z)$ and $\widetilde{a}_1(z)$ such that
       	\[
       		\widehat{\theta}(z)\;=\;\widetilde{a}_0(z)\;\phi_0(z)\;+\;\widetilde{a}_1(z)\;\phi_1(z).
       	\]
       	A possible choice is
       	\[
       		\widetilde{a}_0(z)\;=\; -\frac{9903400\,z-45544275}{466373376\,z^2},
       	\]
       	\[
       	 	\widetilde{a}_1(z)\;=\;\frac{21603855\,z-46560721}{466373376\,z^2}.
       	        \]
        After symmetrization, the resulting mask is such that $a_i=0$ for $i\notin [-14,15]$.\\
       	To obtain a  smaller symmetric mask, we search for a suitable $H_{0,1}(z)$ so that replacing
       	\[
       		\widehat{a}_0(z)\;=\;\widetilde{a}_0(z)\;+\;H_{0,1}(z)\;\phi_1(z),
       	\]
       	\[
       		\widehat{a}_1(z)\;=\;\widetilde{a}_1(z)\;-\;H_{0,1}(z)\;\phi_0(z),
       	\]
        in the previous expressions of ${a}_0(z)$ and ${a}_1(z)$, leads to the final symbol
       	\[
       		a(z)\;=\;a_0(z^3)\;+\;a_1(z^3)\;z\;+\;a_2(z^3)\;z^2,
       	\]
       	satisfying $a(z)=za(z^{-1})$. The choice of $H_{0,1}(z)$ that leads to the shortest mask is
       	\[
       		H_{0,1}(z)\;=\;-\;\frac{16567}{5465313\;z^3}\;-\;\frac{844799}{5465313\;z^2},
       	\]
        and the first half of the resulting symmetric mask $\ba$ is
       	\begin{equation}\label{eq:m3_d6_6pt}
       		\begin{array}{c}
       		\displaystyle	\left \{\; \frac{16567}{466373376},\; 0,\; -\frac{414175}{233186688},\; \frac{224821}{66624768},\; \frac{3}{256},\; \frac{589847}{33312384},\right.\\ \\
       \displaystyle \left.-\frac{83995}{2776032},\; -\frac{25}{256},\; -\frac{2042857}{22208256},\; \frac{1290971}{8328096},\; \frac{75}{128},\; \frac{63152905}{66624768}\;\right \}.
       		\end{array}
       		%,\; \frac{63152905}{66624768},\; \frac{75}{128},\; \frac{1290971}{8328096},\; -\frac{2042857}{22208256},\; -\frac{25}{256},\; -\frac{83995}{2776032},\; \frac{589847}{33312384},\; \frac{3}{256},\; \frac{224821}{66624768},\; -\frac{414175}{233186688},\; 0,\; \frac{16567}{466373376}\;\right]       	
       	\end{equation}
       	The basic limit function $\varphi$ related to the mask in \eqref{eq:m3_d6_6pt} is shown in Figure \ref{fig:m3_d6_6pt}, and two examples of interpolating curves can be found in Figure \ref{fig:m3_d6_6pt_interp}. We have that $\textrm{supp}(\varphi)=[-23/4,23/4]$ and, via joint spectral radius techniques \cite{MR3886713,MR3009529,THOMAS2}, one can prove that $\varphi\in\mathcal{C}^{3.0065}(\mathbb{R})$. By construction the corresponding subdivision scheme reproduces polynomials of degree $5$. On the other hand the primal interpolating ternary $6$-point scheme (see, e.g., \cite{MR3702925}) reproduces quintic polynomials as well but it has a $\mathcal{C}^{2.8300}(\mathbb{R})$ basic limit function supported in $[-4,4]$.
     \end{example}

	\begin{figure}[h!]
		\centering
		\includegraphics[scale=0.75]{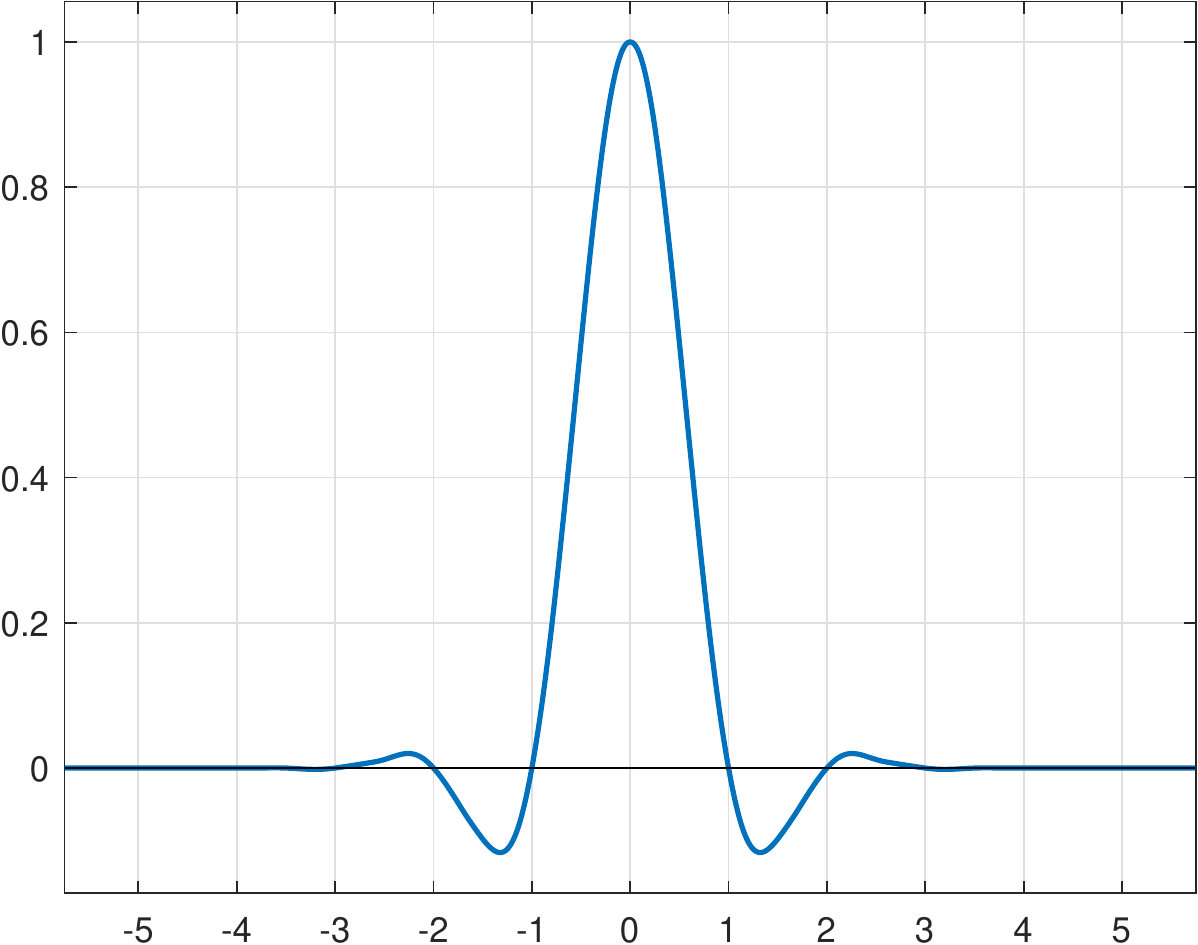}
		\caption{The graph of the basic limit function $\varphi$ related to the mask in \eqref{eq:m3_d6_6pt}.}
		\label{fig:m3_d6_6pt}
	\end{figure}

\bigskip
\bigskip

	\begin{figure}[h!]
		\centering
		\begin{minipage}{0.49\textwidth}
			\includegraphics[scale=0.5]{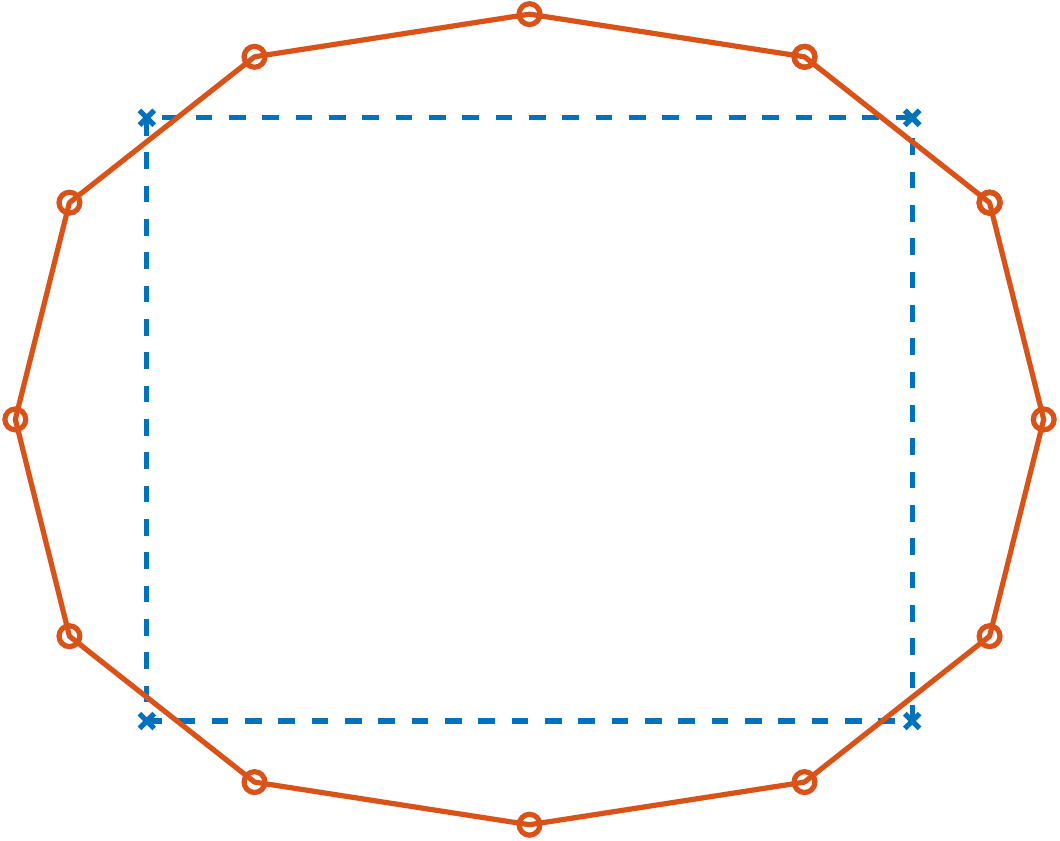}
		\end{minipage}
		\begin{minipage}{0.49\textwidth}
			\includegraphics[scale=0.5]{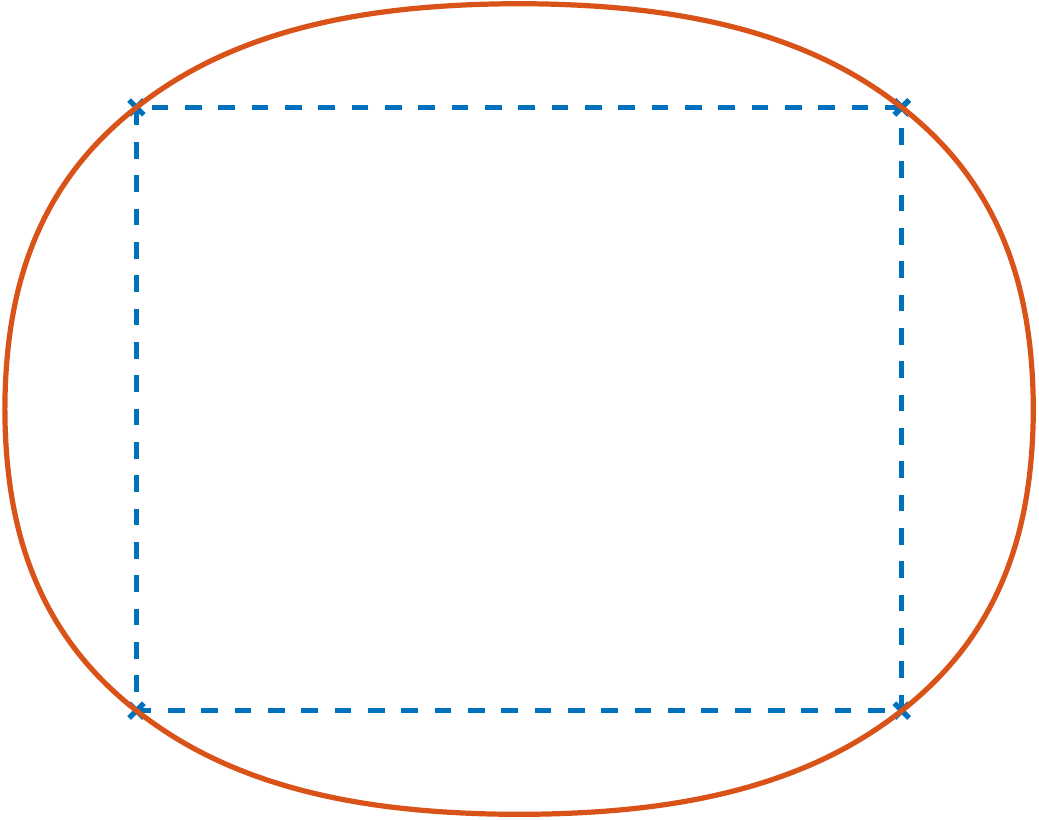}
		\end{minipage} \\ $ $\\
		\begin{minipage}{0.49\textwidth}
			\includegraphics[scale=0.5]{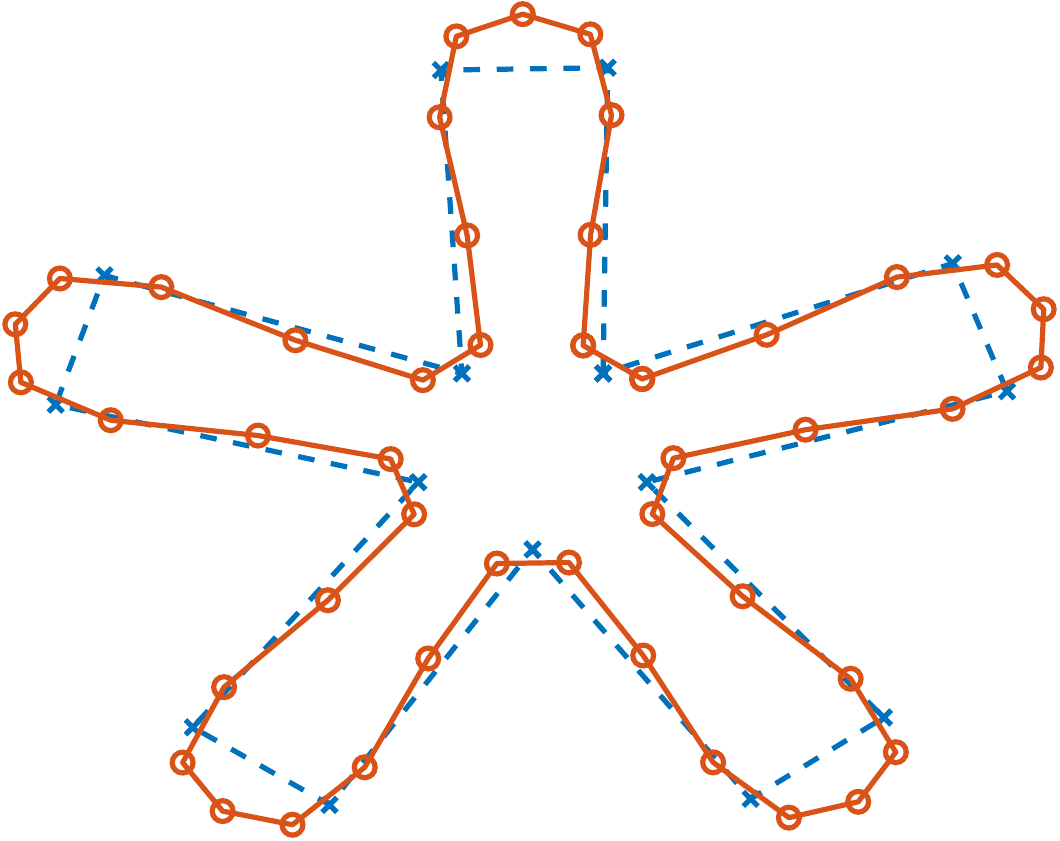}
		\end{minipage}
		\begin{minipage}{0.49\textwidth}
			\includegraphics[scale=0.5]{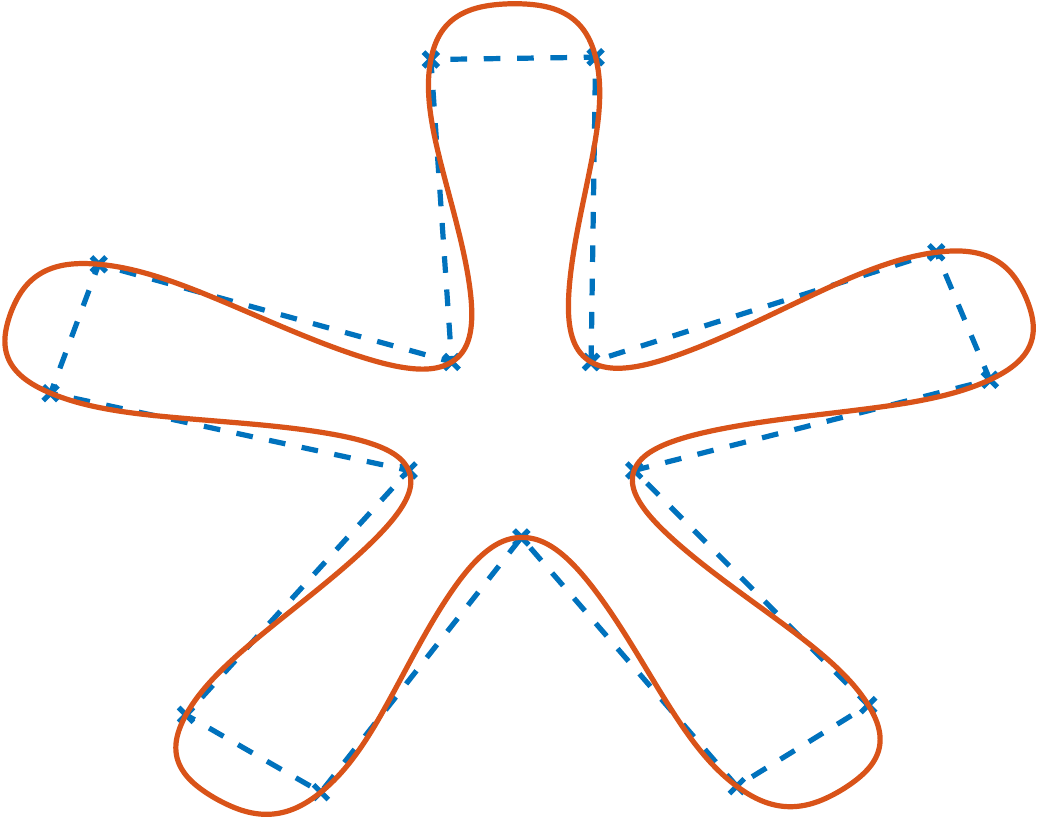}
		\end{minipage}
		\caption{Two examples of interpolating curves given by the subdivision scheme associated to the mask in \eqref{eq:m3_d6_6pt}. On the left, the first level of subdivision starting with the dotted control polygons; on the right, the corresponding interpolating limit curves.}
		\label{fig:m3_d6_6pt_interp}
	\end{figure}

\subsection{The even arity case}
Let us now consider the solution of the linear system $(35)$ in \cite{RV} for the case where
$m=2\ell$ is an even integer.   The system is defined as follows:
\begin{equation}\label{syse}
  M \B a\;=\;\B c, \qquad M\;=\;(\mu_{i,j})_{i,j\in \mathbb Z}, \qquad  \B c \;=\;(c_i)_{i\in \mathbb Z}
  \end{equation}
  where
  \[
  \mu_{i,j}\;=\;\left\{\begin{array}{cl}
  \varphi \left(\displaystyle\frac{i+1}{2}-j \right), & \mbox{ if } i\in m\mathbb Z,\\ \\
  0, & \mbox{ otherwise,}
  \end{array}\right. \qquad c_i\;=\;\left\{\begin{array}{cl}
  1,  & \mbox{ if } i=0,\\ \\
  \varphi \left(\displaystyle\frac{i}{2m} \right),  & \mbox{ if } i\in m(2\mathbb Z+1),\\ \\
   0, & \mbox{ otherwise,}
  \end{array}\right.
  \]
  and $\varphi\colon (2\mathbb Z+1)/2 \rightarrow \mathbb R$ is  a given fixed function. By  suppressing zero rows in
  both $M$ and $\B c$ we obtain  the equivalent linear system
 \begin{equation}\label{sys1}
   \widehat M \B a\;=\;\widehat {\B c}, \qquad
   \widehat M\;=\;(\widehat \mu_{i,j})_{i,j\in \mathbb Z}, \qquad  \widehat {\B c} \;=\;(\widehat c_i)_{i\in \mathbb Z}
  \end{equation}
  where
  \[
  \widehat  \mu_{i,j}\;=\;
  \varphi \left(\displaystyle\frac{im+1}{2}-j \right), \quad i,j\in \mathbb Z,  \qquad \widehat c_i\;=\;\left\{\begin{array}{cl}
  1,  & \mbox{ if } i=0,\\ \\
  \varphi \left(\displaystyle\frac{i}{2} \right), & \mbox{ if }  \mod(i,2)=1,\\ \\
   0, & \mbox{otherwise.}
  \end{array}\right.
  \]
  According to \cite{RV},  \eqref{syse} and \eqref{sys1} can be  expressed in functional form as
  \[
  \sum_{\alpha \in m \mathbb Z} \sum_{\beta \in \mathbb Z} a_{\beta} \varphi \left(\frac{\alpha +1}{2}-\beta \right) z^\alpha\;=\;1 + \sum_{\alpha\in m(2 \mathbb Z+1)}\varphi \left(\frac{\alpha}{2m} \right) z^\alpha
  \]
 which can  be rewritten as
  \begin{equation}\label{eqqq}
  \begin{array}{rcl}
    \sum_{\ell \in \mathbb Z} \sum_{\beta \in \mathbb Z} a_{\beta} \varphi \left(\frac{m \ell +1}{2}-\beta \right) z^\ell&=&1 + \sum_{\ell\in \mathbb Z}\varphi \left(\frac{2\ell+1}{2} \right) z^{2\ell +1}\\ \\
    &=& 1 + \sum_{\ell\in \mathbb Z}\varphi(\ell +\frac{1}{2})z^{2\ell +1} .
  \end{array}
  \end{equation}
  By  Assumption  \ref{a1} the right-hand  side  of \eqref{eqqq}    satisfies
  \[
  \begin{array}{rcl}
  1 + \displaystyle \sum_{\ell\in \mathbb Z}\varphi \left(\ell +\frac{1}{2} \right) z^{2\ell +1}&=&1 + z \phi(z^2)\\
  &=&(z+1)^d(-1)^d\gamma(-z)\\ \\
  &=&(z+1)^d \widetilde \gamma(z), \smallskip \qquad \widetilde \gamma(z)\in \mathbb R[z,z^{-1}].
  \end{array}
  \]
  Concerning the  representation of the left-hand side  of \eqref{eqqq}   let us introduce  the modified subsymbols  defined by
  \begin{equation} \label{eq:phi_hat}
  	\widehat \phi_\ell(z) \;=\;\sum_{i\in \mathbb Z} \varphi \left(\frac{m i +1}{2}-\ell \right) z^i, \qquad 0\leq \ell \leq m-1.
  \end{equation}
  Notice that if $\phi_\ell(z)$, $0\leq \ell \leq m/2-1$,  denote the subsymbols of the mask of arity $m/2$ then  we have
  \begin{equation}\label{halfrel}
  \widehat \phi_\ell(z)\;=\;\phi_\ell(z), \qquad \widehat \phi_{\ell+m/2}(z)=z \widehat \phi_\ell(z), \qquad 0\leq \ell \leq m/2-1.
  \end{equation}
      In particular this implies that
        %\begin{equation}\label{newsign}
        \[
        \widehat \phi_{\ell+m/2}(-1)\;=\;- \widehat \phi_\ell(-1), \qquad \widehat \phi_{\ell+m/2}(1)\;=\; \widehat \phi_\ell(1),  \qquad 0\leq \ell \leq m/2-1.
        \]
        %\end{equation}
        Moreover from    $1 + z \phi(z^2) =  (z+1)^d \widetilde \gamma(z)$ and  $1 - z \phi(z^2) =  (z-1)^d  \gamma(z)$  one  deduces  that
        \begin{equation}\label{signnew1}
          (z+1)^d \widetilde \gamma(z)\;=\;2- (z-1)^d  \gamma(z).
        \end{equation}
  Then  for the left-hand side   of   \eqref{eqqq}  it holds
  \[
  \sum_{\ell \in \mathbb Z} \sum_{\beta \in \mathbb Z} a_{\beta} \varphi \left(\frac{m \ell +1}{2}-\beta \right) z^\ell\;=\;a_0(z^2) \widehat \phi_0(z)+\ldots +a_{m-1}(z^2) \widehat \phi_{m-1}(z).
  \]
  Hence, it follows that  relation \eqref{eqqq} can be reformulated  as  the Bezout-like polynomial equation
  \begin{equation}\label{bezeqqq}
    a_0(z^2) \widehat \phi_0(z)+\ldots +a_{m-1}(z^2) \widehat \phi_{m-1}(z)\;=\;  (z+1)^d\widetilde \gamma(z).
  \end{equation}
  From \eqref{halfrel}  it follows that the equation \eqref{bezeqqq} can be rewritten in a  more customary form  as
  \[
  (a_0(z^2)  + z a_{m/2}(z^2)) \phi_0(z) +\ldots+(a_{m/2-1}(z^2)  + z a_{m-1}(z^2))\phi_{m/2-1}(z)=(z+1)^d\widetilde \gamma(z).
  \]

  Hereafter,  let us assume that Assumption \ref{a2} holds for the subsymbols
  $\phi_\ell(z)$  of the mask of arity $m/2$.
  We also make the following further assumption.
  \begin{assumption}\label{a3}  It is assumed that
     \[
    \phi_0(1)\;=\; \phi_1(1)\;=\;\ldots\;=\;\phi_{m/2-1}(1)\;=\;2/m.
     \]
     \end{assumption}
 \begin{remark}
 	Note that Assumption \ref{a3} is equivalent to impose the matrix $M$ in \eqref{syse} to have eigenvalue $2/m$ with corresponding left-eigenvector $\mathbf{1}$. This ensures that $\sum_{k\in\mathbb{Z}}a_k=m$, which is a necessary condition for the resulting subdivision scheme to be convergent.
 \end{remark}
 Then the solution of  equation  \eqref{bezeqqq} can be found similarly with the odd case.
 Specifically, at the first step  the  unknowns $a^{(s)}_i(1)$, $0\leq i\leq m-1$, $s=1,\ldots d-1$,  are computed
 by solving  a Vandermonde linear system.
The system is formed as follows.  The first $m-1$  equations are
 obtained by differentiation
 of  \eqref{addit} complemented with   relation \eqref{bezeqqq}. The last equation is  found by imposing the property \eqref{signnew1}  on the
 left hand-side of \eqref{bezeqqq}.
If  $\xi_k=e^{2\pi \textrm{i} k/m}$, $1\leq k\leq m$, denote the $m$-th roots of unity  then the system is  of the form
\begin{equation} \label{eq:lin_sys_2}
	\scalebox{0.9}{$m^s\diag\left(\xi_1^{(m-1)s}, \ldots, \xi_{m-1}^{(m-1)s}, (2/m)^{s+1}\right) \mathcal V(\xi_1, \ldots \xi_{m} )\left[ a^{(s)}_0(1), \ldots, a^{(s)}_{m-1}(1)\right]^T\;=\;\B b_s,$}
\end{equation}
where $\diag(\B v)$ is the diagonal matrix with diagonal entries $v_k$ and  $\mathcal V(\xi_1, \ldots, \xi_{m-1},\xi_{m})$ is the Vandermonde matrix with nodes $\xi_k$.   The structure of the
last row follows from  Assumption \ref{a3}.
Since $\xi_k$,  $1\leq k\leq m$, are distinct and non-zero,  the coefficient matrix is nonsingular and $a^{(s)}_i(1)$, $0\leq i\leq m-1$, are uniquely
determined.
Once these quantities are determined
then  the sub-symbols can be represented as follows
    \begin{equation} \label{eq:a_cut}
    \begin{array}{c}
    	a_i(z)\;=\;\widecheck{a}_i(z) + (z-1)^d \widehat a_i(z), \qquad  0\leq i\leq m-1,\\ \\
    	\displaystyle \widecheck{a}_i(z)\;=\;1+ \sum_{j=1}^{d-1}\frac{a^{(j)}_i(1)}{j!} (z-1)^j,
    \end{array}
    \end{equation}
    for suitable   $\widehat a_i(z)\in \mathbb R[z, z^{-1}]$.
   This representation is exploited in the second step to find  a solution of \eqref{bezeqqq}.  If we set
     \begin{equation} \label{eq:theta_even}
     	\theta(z)\;=\;(z+1)^d\widetilde \gamma(z)-\sum_{j=0}^{m-1} \widecheck{a}_i(z^2) \widehat \phi_j(z),%\left( 1+\sum_{\ell=1}^{d-1}\frac{a^{(\ell)}_j(1)}{\ell!} (z^2-1)^\ell\right),
     \end{equation}
   by using  similar arguments as in the proof of Theorem  \ref{theo2}, it is shown that
     \begin{equation} \label{eq:theta_hat_even}
     \theta(z)\;=\;(z^2-1)^d \widehat \theta(z), \qquad \widehat \theta(z)\in \mathbb R[z,z^{-1}].
     \end{equation}
In this  way  equation \eqref{bezeqqq}  can be  simplified as follows
     \[
     \widehat a_0(z^2) \widehat \phi_0(z)+\ldots +\widehat a_{m-1}(z^2) \widehat \phi_{m-1}(z)\;=\;  \widehat \theta(z).
     \]
     which yields the  reduced analogue of \eqref{halfrel}
     \[
     (\widehat a_0(z^2)  + z \widehat a_{m/2}(z^2)) \phi_0(z) +\ldots+(\widehat a_{m/2-1}(z^2)  + z \widehat a_{m-1}(z^2))\phi_{m/2-1}(z)\;=\;\widehat \theta(z).
     \]
     By setting $\widetilde a_i(z)=\widehat a_i(z^2)  + z \widehat a_{i+m/2}(z^2)$, $0\leq i\leq m/2-1$,  we deduce  that the equation
     \begin{equation}\label{lasteq}
     	\widetilde a_0(z)  \phi_0(z) +\ldots+ \widetilde a_{m/2-1}(z)\phi_{m/2-1}(z)\;=\;\widehat \theta(z)
     \end{equation}
     is solvable and  every solution can be written as
     \[
     \overline a_i(z)\;=\;\widetilde a_i(z)  +\sum_{j=i+1}^{m/2-1}H_{i,j}(z)\phi_j(z) -\sum_{j=0}^{i-1}H_{j,i}(z)\phi_j(z),
     \]
     where $\widetilde a_i(z)$  satisfy  \eqref{lasteq}
     and $H_{i,j}(z)$  is any element of $\mathbb R[z,z^{-1}]$.

     Similarly with the odd case it can be shown that  $a(z)$ is a solution if and only if $za(z^{-1})$  is a solution, too. By linearity  this implies that
     $(a(z)+za(z^{-1}))/2$  also  determines a symmetric solution. If this solution is not of minimal length one can exploit the general form above to further
     compress the  representation.
	\begin{remark}
		  For $m=2$  equation \eqref{bezeqqq}  becomes
                  \[
                  (a_0(z^2)  + z a_{1}(z^2)) \phi(z)\;=\;1+z  \phi(z^2)
                  \]
                  which implies  that the first and the last non-zero elements of $a(z)$ must be equal to 1. There follows that the associated subdivision scheme can not be convergent \cite{RV}.
	\end{remark}
     \begin{example}
       	Let us illustrate  our composite approach  for the even  case  by means of a computational example.
		%%%%%%%%%%%%%%%%%%%%%%%%%%%%%%%%%%%%%%%%%%%%%%%%%%%
		We choose $m=4$, $d=6$ and again
		\[
			\varphi\left(\frac{1}{2}+\ell\right)\;=\;\left\{\begin{array}{cl}
				\frac{3}{256}, & \textrm{if } \ell\in\{-3,2\},\\ \\
				-\frac{25}{256}, & \textrm{if } \ell\in\{-2,1\},\\ \\
				\frac{75}{128}, & \textrm{if } \ell\in\{-1,0\},\\ \\
				0, & \textrm{ otherwise.}	
			\end{array}		\right.
		\]	
		Thus, in view of \eqref{eq:phi_hat} and \eqref{halfrel},
		\[
			\widehat{\phi}_0(z)\;=\;\phi_0(z)\;=\;-\;\frac{25}{256\,z}\;+\;\frac{75}{128}\;+\;\frac{3\,z}{256},
		\]
		\[
			\widehat{\phi}_1(z)\;=\;\phi_1(z)\;=\;\frac{3}{256\,z}\;+\;\frac{75}{128}\;-\;\frac{25\,z}{256},
		\]
		\[
			\widehat{\phi}_2(z)\;=\;z\phi_0(z)\;=\;-\;\frac{25}{256}\;+\;\frac{75\,z}{128}\;+\;\frac{3\,z^2}{256},
		\]
		\[
			\widehat{\phi}_3(z)\;=\;z\phi_1(z)\;=\;\frac{3}{256}\;+\;\frac{75\,z}{128}\;-\;\frac{25\,z^2}{256}.
		\]
		After solving the linear system \eqref{eq:lin_sys_2}, from \eqref{eq:a_cut} we obtain
		\[
			a_i(z)\;=\;\widecheck{a}_i(z) + (z-1)^d \widehat a_i(z), \qquad  0\leq i\leq 3,
		\]
		with
		\[
			\widecheck a_0(z)\;=\;1+\frac{\left(z-1\right)}{8}-\frac{7\,{\left(z-1\right)}^2}{128}+\frac{35\,{\left(z-1\right)}^3}{1024}-\frac{805\,{\left(z-1\right)}^4}{32768}+\frac{4991\,{\left(z-1\right)}^5}{262144},%+(z-1)^6\widehat{a}_0(z),
		\]
		\[
			\widecheck a_1(z)\;=\;1-\frac{\left(z-1\right)}{8}+\frac{9\,{\left(z-1\right)}^2}{128}-\frac{51\,{\left(z-1\right)}^3}{1024}+\frac{1275\,{\left(z-1\right)}^4}{32768}-\frac{8415\,{\left(z-1\right)}^5}{262144},%+(z-1)^6\widehat{a}_1(z),
		\]
		\[
			 \widecheck a_2(z)\;=\;1-\frac{3\,\left(z-1\right)}{8}+\frac{33\,{\left(z-1\right)}^2}{128}-\frac{209\,{\left(z-1\right)}^3}{1024}+\frac{5643\,{\left(z-1\right)}^4}{32768}-\frac{39501\,{\left(z-1\right)}^5}{262144},%+(z-1)^6\widehat{a}_2(z),
		\] 	
		\[
			\widecheck a_3(z)\;=\;1-\frac{5\,\left(z-1\right)}{8}+\frac{65\,{\left(z-1\right)}^2}{128}-\frac{455\,{\left(z-1\right)}^3}{1024}+\frac{13195\,{\left(z-1\right)}^4}{32768}-\frac{97643\,{\left(z-1\right)}^5}{262144}.%+(z-1)^6\widehat{a}_3(z).
		\]
		To search for compatible $\widehat{a}_0(z)$, $\widehat{a}_1(z)$, $\widehat{a}_2(z)$ and $\widehat{a}_3(z)$, we first compute
		\[
			\widehat{\theta}(z)\;=\;\frac{3}{256\,z^5}\;-\;\frac{7}{256\,z^3}\;+\;\frac{5086563}{16777216\,z}\;-\;\frac{580643}{16777216}
		\]
		such that, according to \eqref{eq:theta_even} and \eqref{eq:theta_hat_even},
		\[
			(z^2-1)^6\widehat{\theta}(z)\;=\;(z+1)^6\;\widetilde{\gamma}(z)\;-\;\sum_{i=0}^{3}\;\widecheck{a}_i(z^2)\;\widehat{\phi}_i(z),%\Big(a_i(z^2)-(z^2-1)^6\widehat{a}_i(z^2)\Big)
		\]
		with
		\[
			\widetilde{\gamma}(z)\;=\;\frac{3}{256\,z^5}-\frac{9}{128\,z^4}+\frac{19}{128\,z^3}-\frac{9}{128\,z^2}+\frac{3}{256\,z},
		\]
		due to \eqref{signnew1}. Then we search for $\widetilde{a}_0(z)$ and $\widetilde{a}_1(z)$ that solve the reduced Bezout equation in \eqref{lasteq},
		\begin{equation} \label{eq:red_Bez}
			\widehat{\theta}(z)\;=\;\widetilde{a}_0(z)\;{\phi}_0(z)\;+\;\widetilde{a}_1(z)\;{\phi}_1(z).
		\end{equation}
		A possible choice is
		\[
			\widetilde{a}_0(z)\;=\;\frac{2126507351527}{157810688\,z}-\frac{176620228675}{78905344},
		\]
		\[ 	
		 	\widetilde{a}_1(z)\;=\;\frac{1}{z^4}-\frac{50}{z^3}+\frac{2506}{z^2}-\frac{2118539063675}{157810688\,z}-\frac{21194427441}{78905344}.
		\]
		Once we have a solution of \eqref{eq:red_Bez}, we search for
		\[
			\overline a_0(z)\;=\;\widetilde a_0(z)+H_{0,1}(z)\;\phi_1(z),
		\]
		\[
			\overline a_1(z)\;=\;\widetilde a_1(z)-H_{0,1}(z)\;\phi_0(z),
		\]
		so that $\{\widehat{a}_{k}(z)\}_{k=0,\dots,3}$ fulfilling
		\[
			\overline a_i(z)\;=\;\widehat a_i(z^2)  + z\; \widehat a_{i+2}(z^2),\qquad i\in\{0,1\},
		\]
		lead to a symbol $a(z)$ satisfying $a(z)=za(z^{-1})$.
		For example, the choice
		\[\begin{array}{rcl}
			H_{0,1}(z)&=& \displaystyle -\frac{7064809147}{308224\,z} + \displaystyle \frac{281633113}{616448\,z^2} - \displaystyle \frac{2817667}{308224\,z^3} + \displaystyle \frac{119853}{616448\,z^4}\\ \\
			&&\quad + \displaystyle \frac{7302199}{596413440\,z^5} - \displaystyle \frac{3127}{1232896\,z^6} + \displaystyle \frac{947}{1331280\,z^7}			
		\end{array}\]
		leads to
		\[\begin{array}{rcl}
 			\overline{a}_0(z)&=& \displaystyle \frac{39501}{262144\,z} - \displaystyle \frac{4991}{262144\,z^2} - \displaystyle \frac{5643}{262144\,z^3} + \displaystyle \frac{24415849}{4362338304\,z^4}  + \displaystyle \frac{394938757}{40715157504\,z^5}
 \\ \\
 		&&\qquad  - \displaystyle \frac{61600783}{43623383040\,z^6} + \displaystyle \frac{15760091}{40715157504\,z^7} + \displaystyle \frac{947}{113602560\,z^8}
 		\end{array}\]
 		\[\begin{array}{rcl}
 			\overline{a}_1(z)&=&\displaystyle \frac{97643}{262144\,z}+ \displaystyle \frac{8415}{262144\,z^2}- \displaystyle \frac{7917}{262144\,z^3}- \displaystyle \frac{49446367}{7270563840\,z^4} + \displaystyle \frac{482174039}{40715157504\,z^5}
 \\ \\
 	&&\qquad + \displaystyle \frac{116624327}{43623383040\,z^6}- \displaystyle \frac{27054815}{40715157504\,z^7}+ \displaystyle \frac{4735}{68161536\,z^8},
		\end{array}\]
		and so
		\[
			\widehat{a}_0(z)\;=\;-\frac{4991}{262144\,z}+\frac{24415849}{4362338304\,z^2}-\frac{61600783}{43623383040\,z^3}+\frac{947}{113602560\,z^4},
		\]
		\[
			\widehat{a}_1(z)\;=\;\frac{8415}{262144\,z}-\frac{49446367}{7270563840\,z^2}+\frac{116624327}{43623383040\,z^3}+\frac{4735}{68161536\,z^4},
		\]
		\[
			\widehat{a}_2(z)\;=\;\frac{39501}{262144\,z}-\frac{5643}{262144\,z^2}+\frac{394938757}{40715157504\,z^3}+\frac{15760091}{40715157504\,z^4},
		\]
		\[
			\widehat{a}_3(z)\;=\;\frac{97643}{262144\,z}-\frac{7917}{262144\,z^2}+\frac{482174039}{40715157504\,z^3}-\frac{27054815}{40715157504\,z^4}.
		\]
		Replacing the previous expressions in the above equations of $a_0(z)$, $a_1(z)$, $a_2(z)$ and $a_3(z)$ and using
		\[
			a(z)\;=\;a_0(z^4)\;+\;a_1(z^4)\;z\;+\;a_2(z^4)\;z^2\;+\;a_3(z^4)\;z^3,
		\]
		the first half of the resulting symmetric mask $\ba$ is
		\begin{equation}\label{eq:m4_d6_6pt}
			\begin{array}{c}
				\displaystyle\left \{\; \frac{947}{113602560},\; \frac{4735}{68161536},\; \frac{15760091}{40715157504},\; -\frac{27054815}{40715157504},\; -\frac{63782671}{43623383040},\right.\\ \\
				\displaystyle\frac{98441927}{43623383040},\; \frac{42911173}{5816451072},\; \frac{92071847}{5816451072},\; \frac{154804477}{10905845760},\; -\frac{79247347}{3635281920},\\ \\
				\displaystyle-\frac{143318065}{1938817024},\; -\frac{215643011}{1938817024},\; -\frac{71706399}{969408512},\; \frac{4869166267}{43623383040},\; \frac{2428957997}{5816451072},\\ \\
				\displaystyle\left.\frac{4331006815}{5816451072},\; \frac{528433771}{545292288} \;\right \}.
			\end{array}
			%,\; \frac{528433771}{545292288},\; \frac{4331006815}{5816451072},\; \frac{2428957997}{5816451072},\; \frac{4869166267}{43623383040},\; -\frac{71706399}{969408512},\; -\frac{215643011}{1938817024},\; -\frac{143318065}{1938817024},\; -\frac{79247347}{3635281920},\; \frac{154804477}{10905845760},\; \frac{92071847}{5816451072},\; \frac{42911173}{5816451072},\; \frac{98441927}{43623383040},\; -\frac{63782671}{43623383040},\; -\frac{27054815}{40715157504},\; \frac{15760091}{40715157504},\; \frac{4735}{68161536}\\ \frac{947}{113602560}	
		\end{equation}

\bigskip
The basic limit function $\varphi$ related to this mask is shown in Figure \ref{fig:m4_d6_6pt}, and two examples of interpolating curves can be found in Figure \ref{fig:m4_d6_6pt_interp}. We have that $\textrm{supp}(\varphi)=[-11/2,11/2]$ and, via joint spectral radius techniques, one can prove that $\varphi\in\mathcal{C}^{3.0507}(\mathbb{R})$. By construction the corresponding subdivision scheme reproduces polynomials of degree $5$. With respect to the primal interpolating quaternary scheme proposed by Conti et al. \cite{MR2775138}, which has a $\mathcal{C}^{2.2924}(\mathbb{R})$ basic limit function supported in $[-7/3,7/3]$ and reproduces quadratic polynomials only, the mask obtained here is much wider but achieve desirable properties in application such as $\mathcal{C}^3$-smoothness and reproduction of higher degree polynomials.
On the other hand the primal interpolating quaternary $6$-point scheme (see e.g.  \cite{MR3702925}) reproduces quintic polynomials as well but it has a $\mathcal{C}^{2.3198}(\mathbb{R})$ basic limit function supported in $[-11/3,11/3]$.
     \end{example}

	\begin{figure}[h!]
		\centering
		\includegraphics[scale=0.75]{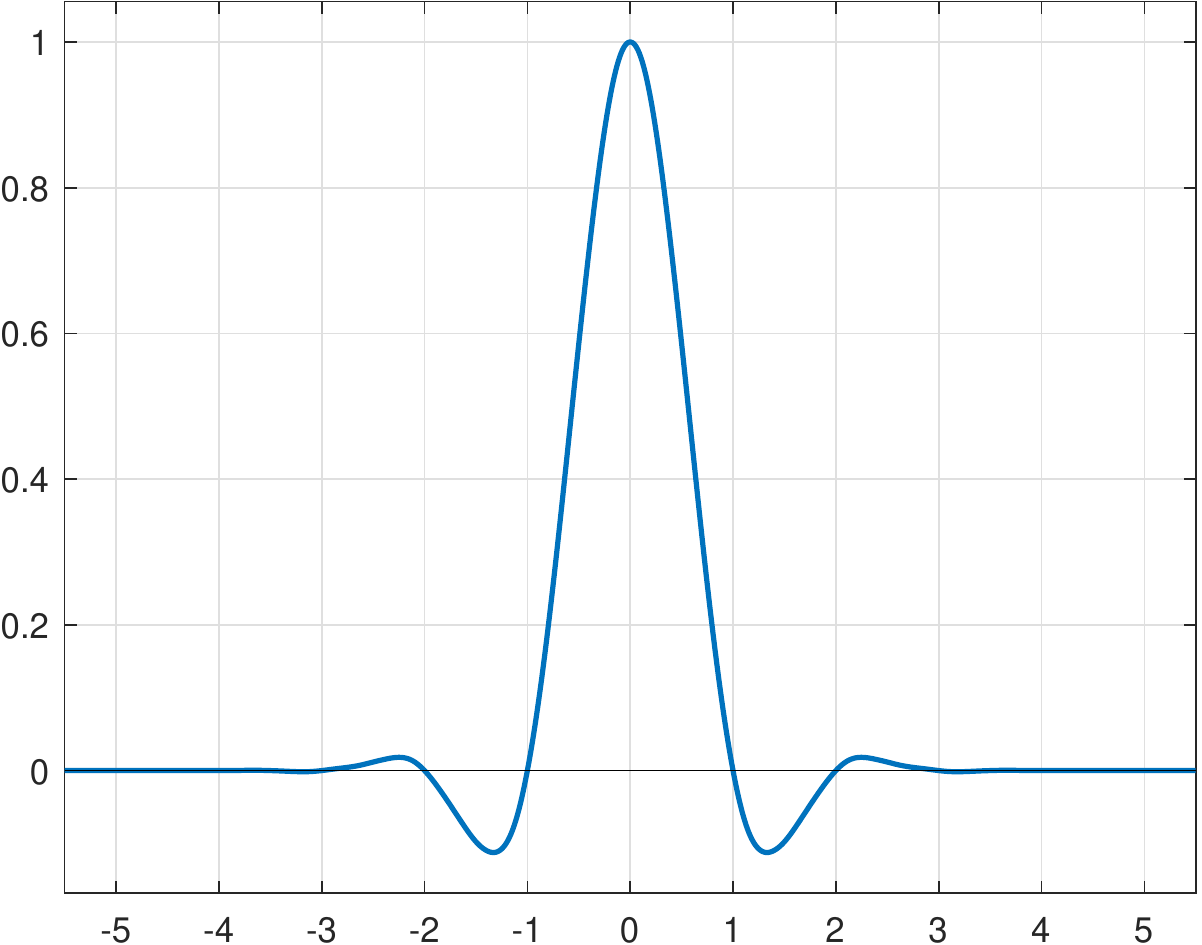}
		\caption{The graph of the basic limit function $\varphi$ related to the mask in \eqref{eq:m4_d6_6pt}.}
		\label{fig:m4_d6_6pt}
	\end{figure}

\bigskip
\bigskip

	\begin{figure}[h!]
		\centering
		\begin{minipage}{0.49\textwidth}
			\includegraphics[scale=0.5]{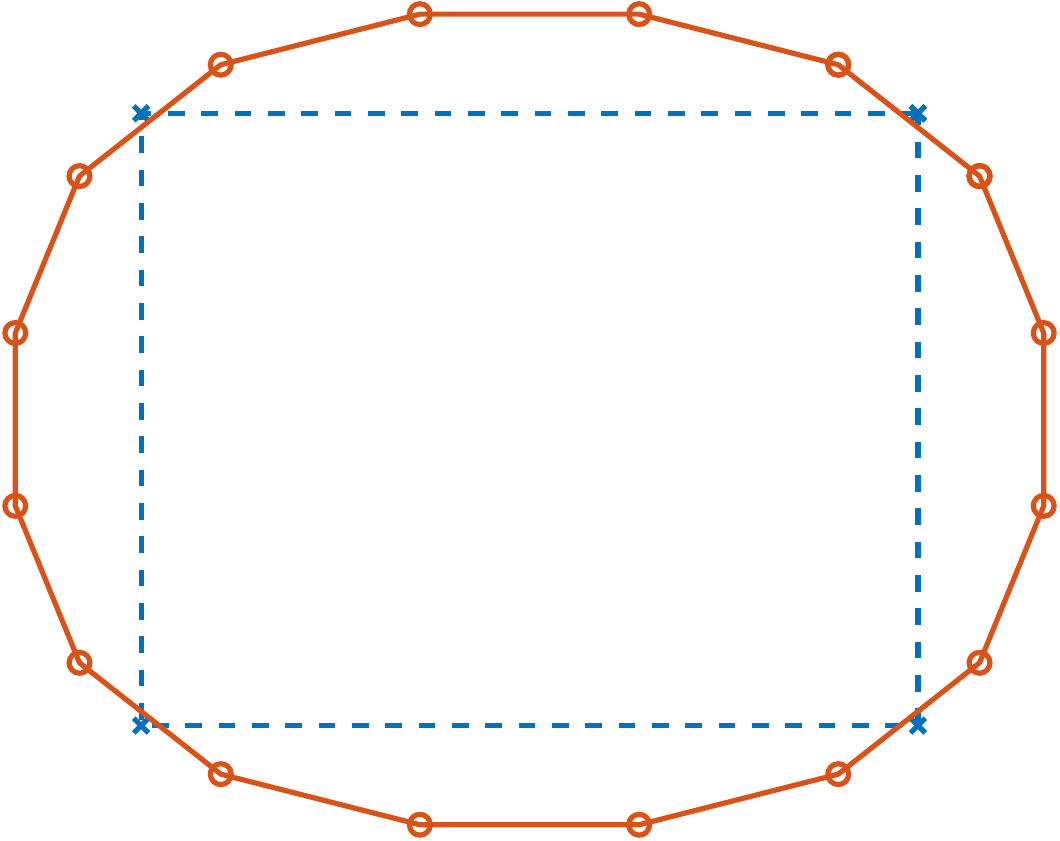}
		\end{minipage}
		\begin{minipage}{0.49\textwidth}
			\includegraphics[scale=0.5]{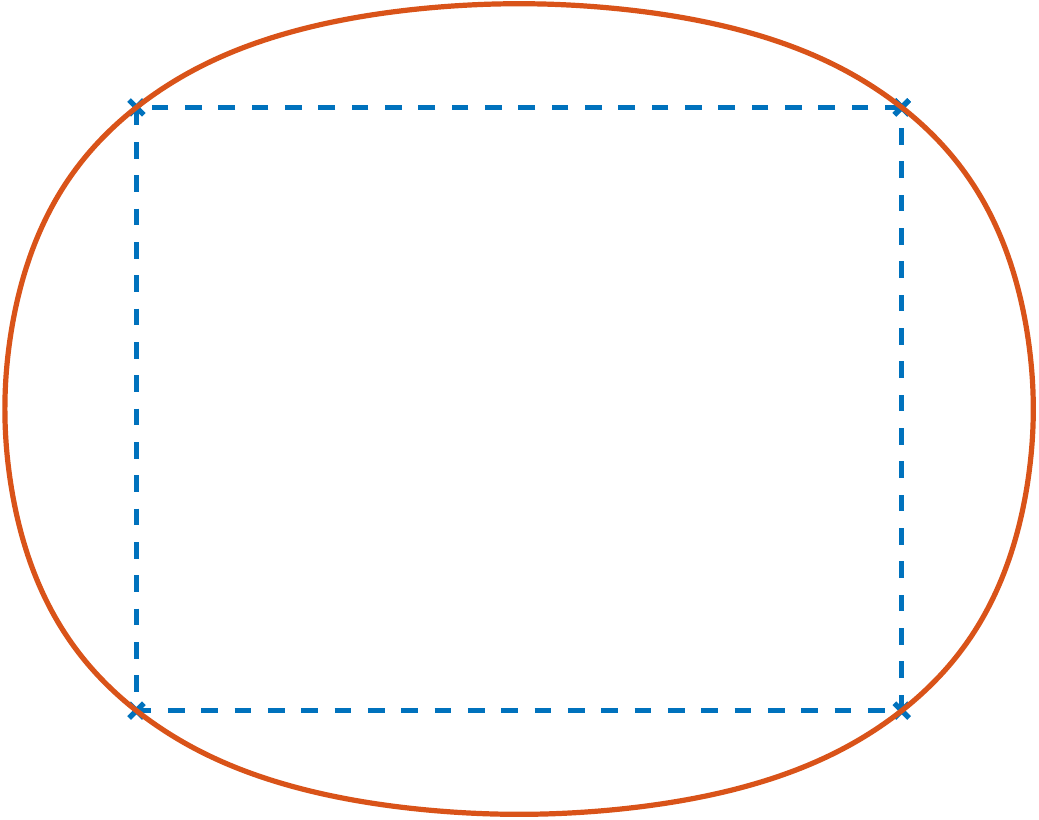}
		\end{minipage} \\ $ $\\
		\begin{minipage}{0.49\textwidth}
			\includegraphics[scale=0.5]{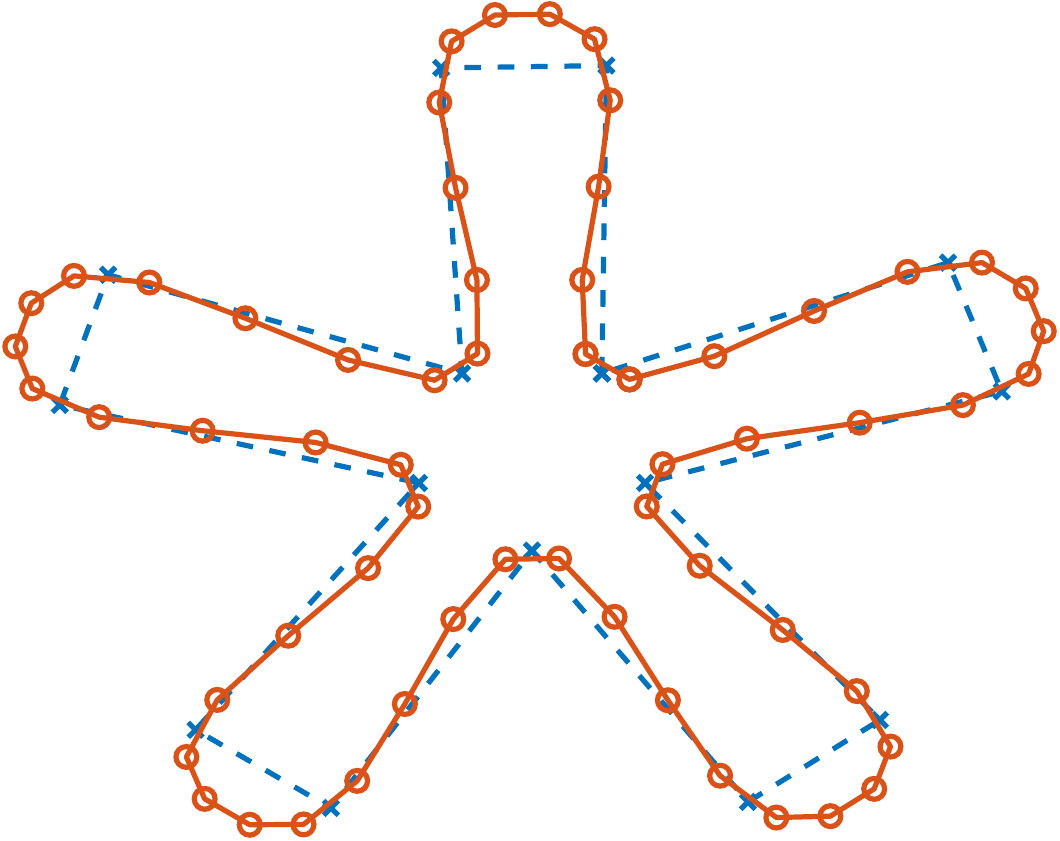}
		\end{minipage}
		\begin{minipage}{0.49\textwidth}
			\includegraphics[scale=0.5]{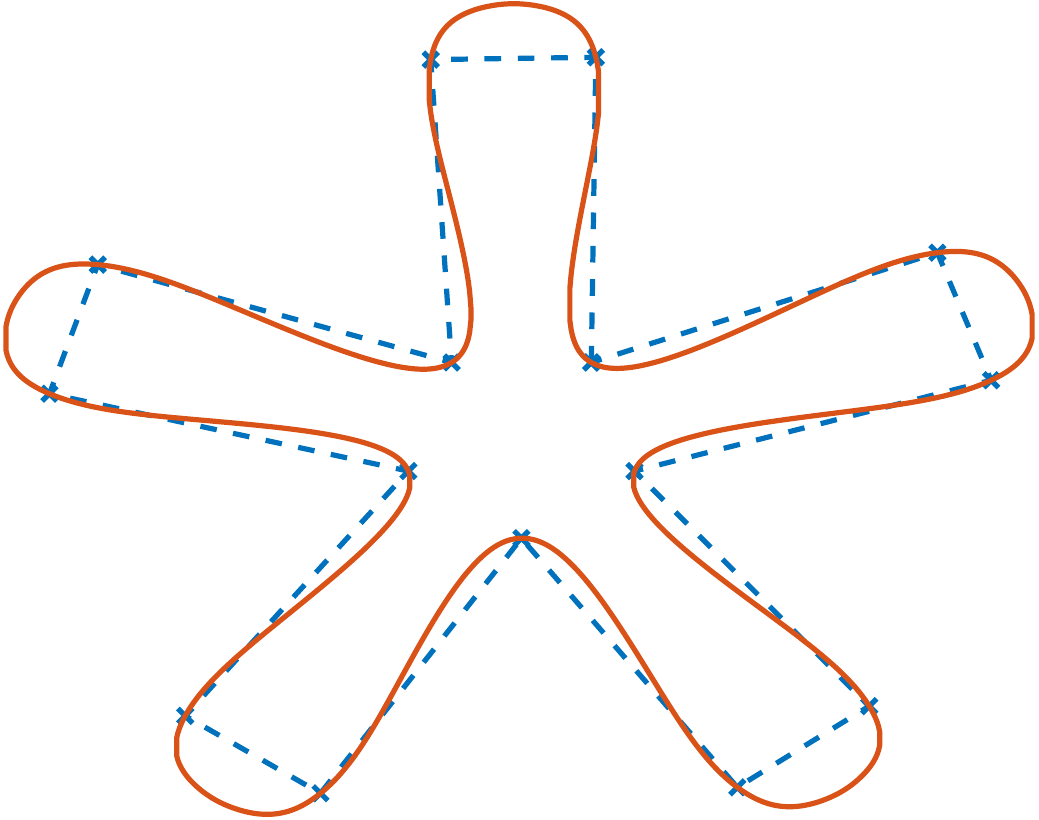}
		\end{minipage}
		\caption{Two examples of interpolating curves given by the subdivision scheme associated to the mask in \eqref{eq:m4_d6_6pt}. On the left, the first level of subdivision starting with the dotted control polygons; on the right, the corresponding interpolating limit curves.}
		\label{fig:m4_d6_6pt_interp}
	\end{figure}

\begin{acknowledgements}
The authors are members of INdAM - GNCS, which partially supported this work.
\end{acknowledgements}

%\bibliographystyle{plain}
%\bibliography{Biblio_dual_interp}

\end{document}